\documentclass[11pt]{article}
\usepackage[T1]{fontenc}
\usepackage{lmodern,amsmath,amssymb,amsthm,mathtools,microtype}
\usepackage[margin=1in]{geometry}
\usepackage[hidelinks]{hyperref}
\usepackage{enumitem}
\setlist[enumerate]{itemsep=3pt,topsep=6pt}
\allowdisplaybreaks
\numberwithin{equation}{section}
\newtheorem{theorem}{Theorem}[section]
\newtheorem{proposition}[theorem]{Proposition}
\newtheorem{lemma}[theorem]{Lemma}

\theoremstyle{definition}

\theoremstyle{remark}

\newcommand{\R}{\mathbb R}

\newcommand{\A}{\mathcal A}
\newcommand{\sgn}{\operatorname{sgn}}
\newcommand{\cof}{\operatorname{cof}}

\newcommand{\Log}{\operatorname{Log}}

\newcommand{\norm}[1]{\lVert #1\rVert}

\title{Auerbach bases of real $\ell_p^n$ spaces}
\author{Arun Maiti \and Debmalya Sain}

\date{}

\begin{document}
\maketitle
\begin{abstract}
We study Auerbach bases of the space $\ell^n_p$, and in particular, prove that  $\ell_p^n$ has only finitely many Auerbach bases for $p>n$ and $1<p<n/(n-1)$. An explicit one-parameter family in $\ell_4^4$
shows that the condition $p>n$ cannot in general
be weakened to $p\geq n$. Further, we provide an explicit classification of the Auerbach bases in dimension two and three.
\end{abstract}

\section{Introduction}

The concept of Auerbach bases emerges naturally in the context of \textit{Birkhoff-James orthogonality}. It is a generalization of the notion of orthogonality in Hilbert 
spaces to Banach spaces introduced by Birkhoff  \cite{Birkhoff1935} and James \cite{James1947}. 
Let $\mathbb{X}$ be a $n$ dimensional real Banach space and $S_{\mathbb{X}}$ denote its unit sphere. 
For $x, y \in \mathbb{X}$, $x$ is said to be orthogonal 
to $y$ in the sense of Birkhoff-James, denoted by $x \perp_B y$, if 
\[|| x|| \leq || x + \lambda y|| \quad \mbox{for all} \quad \lambda \in \mathbb{R}.\]
A basis $\mathcal{B}$ of a Banach space 
$(X, ||.||)$ is called an Auerbach basis if for all $x \in \mathcal{B}$, 
\[ ||x|| = 1 \quad \mbox{and} \ x \perp_B \operatorname{span}(\mathcal{B}\setminus\{x\}).\]

Equivalently,  \emph{Auerbach basis} of an \(n\)-dimensional normed space $\mathbb{X}$ is a
basis \((b_1,\ldots,b_n)\) whose biorthogonal functionals
\((b_1^*,\ldots,b_n^*)\) satisfy
\[
 \norm{b_j}=\norm{b_j^*}=1\qquad(1\leq j\leq n).
\]

Every finite-dimensional normed space has such a basis. Auerbach's
unpublished existence result \cite{Auerbach1930} was followed by  Day and Taylor published results \cite{Day1947,Taylor1947}  using
volume arguments. Plichko's volume method gave general multiplicity results
\cite{Plichko1996}. Weber and Wojciechowski obtained lower bounds
by interpreting Auerbach bases as critical configurations of the
determinant and applying topological methods \cite{WeberWoj2017}. Guerra–Jana–Maiti showed that an improved can be obtained by computing the mod-\(2\) cohomology of unordered flag manifolds  \cite{GM25}.

For the concrete spaces \(\ell_p^n\), the opposite question---when
the set of bases is finite---requires control of every critical
configuration, including degenerate ones. 

Throughout the paper the spaces and matrices are real. Let \(B\) be the
\(n\times n\) matrix with rows \(b_1,\ldots,b_n\), and write
\(\A_{n,p}\) for the set of all such ordered normalized Auerbach matrices
in \(\ell_p^n\). Its cardinality counts matrices without taking a
quotient. For classification purposes, we call two matrices
\emph{equivalent} if one is obtained from the other by signed
permutations of the rows (reordering and changing signs of the basis
vectors) and signed permutations of the columns (the coordinate
isometries of \(\ell_p^n\)). Put \(q=p/(p-1)\) and define
\[
 J_p(x)_i=\sgn(x_i)|x_i|^{p-1},\qquad J_p(0)_i=0,
\]
applying \(J_p\) to matrices entrywise. If \(\norm{x}_p=1\), then
\(J_p(x)\) is its unique norm-one supporting functional.
Consequently,
\begin{equation}\label{eq:auerbach}
 B\in\A_{n,p}\quad\Longleftrightarrow\quad
 BJ_p(B)^T=I_n
 \quad\Longleftrightarrow\quad J_p(B)=B^{-T}.
\end{equation}
In particular \(\A_{n,p}\) is compact, and \(B\mapsto J_p(B)\)
bijectionally identifies \(\A_{n,p}\) with \(\A_{n,q}\).

Our first result gives an explicit interval on which finiteness
holds uniformly.
\begin{theorem}\label{thm:main}
For each integer \(n\geq2\) there is a finite constant \(C_n\) such that
\[
 \#\A_{n,p}\leq C_n
 \qquad\text{if}\qquad
 p\in\left(1,\frac{n}{n-1}\right)\cup(n,\infty).
\]
\end{theorem}
The strict inequality \(p>n\) cannot simply be replaced by
\(p\geq n\). We exhibit a nonconstant family in \(\A_{4,4}\) in \S \ref{sec:infinite 4dim}.
Adjoining coordinate vectors gives infinite families in
\(\A_{n,4}\) for every \(n\geq4\), including \(2<p<n\) when \(n>4\).
At \(p=2\), the set \(\A_{n,2}=O(n)\) is infinite for \(n\geq2\).

The proof combines o-minimal geometry with a complex analytic isolation
argument. First, the family \eqref{eq:auerbach} is definable in the real
exponential field $\R_{\exp}$, which is o-minimal
\cite{Wilkie1996, vdD1998}. Standard finiteness results
\cite[Ch.~3]{vdD1998}\cite[Th.~0.3]{KPS1986}
bound the number of connected components uniformly in $p$.
Differentiating the row normalizations along definable paths shows that
$\det B$ is constant on each component. This proves the uniform
finiteness of determinant values in Proposition~\ref{prop:volume-rigidity}.
For a fixed value $\delta=\det B\ne0$, the equations reduce to
\[
 J_p(B)=\delta^{-1}\cof B,
\]
whose right-hand side has degree $n-1$.

When $p>n$, Lemma~\ref{lem:real-power-finiteness} shows that the solutions
of this system are isolated, using local holomorphic extensions and
the subharmonic maximum principle. O-minimal finiteness then gives
a uniform bound on their number, completing the proof for $p>n$.
Duality gives the range $1<p<n/(n-1)$.

In \S \ref{classification}, we  further provide a complete classification of Auerbach bases in dimensions two and three using elementary analysis.

\section{Finiteness of Auerbach bases}

We first prove the finiteness of the connected components of the set $\A_{n,p}$.

\begin{proposition} \label{prop:volume-rigidity}
For each \(n\), there exists an integer \(N_n\) such that, for every
\(p>1\), the set \(\mathcal A_{n,p}\) has at most \(N_n\) connected
components, and the determinant is constant on each connected component.

\end{proposition}

\begin{proof}
We define the total family

\[ \mathcal X_n
=\{(p,B)\in(1,\infty)\times\R^{n\times n}:
    BJ_p(B)^T=I_n\}. \]
    
 $\mathcal X_n$ is definable in $\R_{\exp}$.
Indeed, for $x\ne0$,
\[
 \sgn(x)|x|^{p-1}
 =\sgn(x)\exp\bigl((p-1)\log|x|\bigr),
\]
and its value at zero is specified separately.
The structure $\R_{\exp}$ is o-minimal
 \cite{Wilkie1996, vdD1998}.

Each fiber \((\mathcal X_n)_p=\mathcal A_{n,p}\) has finitely many
definable connected components and their number is bounded independently of \(p\) 
\cite[ \S3]{vdD1998}.

Fix \(p>1\) and a connected component \(C\) of \(\mathcal A_{n,p}\).
Any two points of \(C\) can be joined by a definable continuous path
\(B:[0,1]\to C\)
\cite[ Chap. 6]{vdD1998}. The monotonicity theorem
\cite[Chap. 3]{vdD1998}
gives a finite partition on whose open intervals each coordinate of
\(B\) is monotone or constant. A monotone real function is differentiable
almost everywhere; its nondifferentiability set is definable and hence
finite. Applying monotonicity again to the definable derivatives gives
a finite refinement on whose open intervals \(B\) is \(C^1\).

On each such interval, writing \(B=(b_{ij})\), the diagonal equations give
\[
\sum_j |b_{ij}|^p=1,
\qquad\text{hence}\qquad
\sum_{i,j}J_p(b_{ij})\,b_{ij}'=0.
\]
Since \(B^{-1}=J_p(B)^T\), Jacobi's formula yields
\[
\frac{d}{dt}\det B
 =\det B\,\operatorname{tr}(B^{-1}B')
 =\det B\sum_{i,j}J_p(b_{ij})\,b_{ij}'
 =0.
\]
Continuity across the finitely many partition points shows that the
determinant is constant on \(C\). It therefore attains at most \(N_n\)
values on \(\mathcal A_{n,p}\), as required.
\end{proof}

Proposition~\ref{prop:volume-rigidity} shows that
\[
\Delta_{n,p}:=\{\det B:B\in\A_{n,p}\}
\]
has at most $N_n$ elements.
 For
$\delta\in\Delta_{n,p}$, every $B\in\A_{n,p}$ with
$\det B=\delta$ satisfies
\[
 J_p(B)=\delta^{-1}\cof B.
\]
The right-hand side is polynomial in the matrix entries, of
degree at most $n-1$. The following lemma then establishes the finiteness 
of the solutions when $p-1>n-1$.

\begin{lemma}\label{lem:real-power-finiteness}
Let $N\geq 1$ and $d\geq 0$ be integers, and let
$P_1,\ldots,P_N\in\mathbb{R}[x_1,\ldots,x_N]$ have total
degree at most $d$. If $a_1,\ldots,a_N>d$, then the system
\[
    \operatorname{sgn}(x_i)|x_i|^{a_i}=P_i(x),
    \qquad i=1,\ldots,N,
\]
has finitely many solutions in $\mathbb{R}^N$, where
$\operatorname{sgn}(0)=0$. Moreover, their number is bounded
by a constant depending only on $N$ and $d$.
\end{lemma}

\begin{proof}
We first prove that every solution of each local holomorphic
system
\begin{equation}\label{eq:local-phase-system}
    P_i(z)=\eta_i\exp(a_i\Log z_i),
    \qquad i=1,\ldots,N,
    \qquad |\eta_i|=1,
\end{equation}
is isolated in $(\mathbb{C}^*)^N$. Here the logarithms are
chosen on a neighbourhood of the point under consideration.

Let $C\geq 1$ bound the sum of the absolute values of the
coefficients of each $P_i$, and put
$\gamma=\min_i(a_i-d)>0$. At any solution of
\eqref{eq:local-phase-system}, set $M=\max_j|z_j|$. 
If $M\geq 1$ and $|z_i|=M$, then
\[
    M^{a_i}=|P_i(z)|\leq CM^d.
\]
 Consequently, every such solution satisfies
\begin{equation}\label{eq:phase-bound}
    |z_j|\leq R:=C^{1/\gamma},
    \qquad j=1,\ldots,N.
\end{equation}

Let $L\subset(\mathbb{C}^*)^N$ consist of the points at
which some system \eqref{eq:local-phase-system} has a
positive-dimensional solution germ. On a connected
logarithm chart $U$, define the holomorphic map
\[
    H:U\longrightarrow\mathbb{C}^N,
    \qquad H_i(z)=P_i(z)\exp(-a_i\Log z_i).
\]
Writing $S^1=\{\eta\in\mathbb{C}:|\eta|=1\}$, we have
\[
    L\cap U
    =
    H^{-1}\bigl((S^1)^N\bigr)
    \cap
    \bigl\{z\in U:\dim_z H^{-1}(H(z))\geq 1\bigr\}.
\]
The local fibre dimension is upper semicontinuous
\cite[Thm~2.7]{Adamus2012}; hence $L$ is relatively
closed in $(\mathbb{C}^*)^N$. This description is independent
of the chosen logarithms: changing a branch multiplies
each $H_i$ by a constant of modulus one.

Suppose that $L$ is nonempty. Consider
\[
    \Phi(z)=\sum_{j=1}^N\log|z_j|
             +\sum_{j=1}^N|z_j|^2.
\]
By \eqref{eq:phase-bound}, $\Phi$ is bounded above on $L$.
Moreover, within the region $|z_j|\leq R$, it tends to
$-\infty$ whenever any coordinate tends to zero.
Relative closedness therefore implies that $\Phi$ attains
its maximum on $L$, say at $z^*$.

Choose a positive-dimensional irreducible component of
the local fibre through $z^*$. A sufficiently small
representative of this component lies in $L$ and contains
a nonconstant holomorphic curve
$f:\Delta\to(\mathbb{C}^*)^N$ with $f(0)=z^*$.
The subharmonic function $\Phi\circ f$ attains a maximum
at $0$, so it is constant. On the other hand, since each
$f_j$ is nonvanishing,
\[
    \Delta(\Phi\circ f)
       =4\sum_{j=1}^N|f_j'|^2.
\]
Thus $f$ is constant, a contradiction. It follows that
$L=\varnothing$, proving the asserted isolation.

Now fix a nonempty support
$I\subseteq\{1,\ldots,N\}$ and signs
$s_i\in\{-1,1\}$ for $i\in I$. Substitute
$x_i=s_it_i$ with $t_i>0$ for $i\in I$, and $x_j=0$
for $j\notin I$. The equations indexed by $I$ become
\[
    t_i^{a_i}=Q_i(t),\qquad i\in I,
\]
where each $Q_i$ is a real polynomial of total degree
at most $d$. Applying the preceding argument in $|I|$
variables shows that all positive solutions are isolated.

These positive solution sets are definable in the
o-minimal structure $\mathbb{R}_{\exp}$, since
$t^a=\exp(a\log t)$ for $t>0$. A definable set consisting
of isolated points is finite, by cell decomposition
\cite[ \S3]{vdD1998}. The equations indexed
outside $I$ only restrict this finite set further.
There are finitely many support and sign choices, and
the empty support contributes at most the origin.
This proves finiteness.

Finally, for fixed $N$ and $d$, allow all polynomial
coefficients and all exponents $a_i>d$ to vary.
The resulting solution sets form a definable family
in $\mathbb{R}_{\exp}$, every fibre of which is finite.
Uniform finiteness for definable families (\cite{vdD1998, KPS1986}) gives a common bound
depending only on $N$ and $d$.
\end{proof}

\begin{proof}[\textbf{Proof of Theorem~\ref{thm:main}}] \label{proofthm1}

For $p>n$, put $a=p-1>n-1$. By
Proposition~\ref{prop:volume-rigidity}, there are at most $N_n$
determinant values. For each such value $\delta$, apply
Lemma~\ref{lem:real-power-finiteness} with $N=n^2$, $d=n-1$ and
$P_{ij}(B)=\delta^{-1}(\cof B)_{ij}$.
Its uniform bound $K_{n^2,n-1}$ is independent of $p$ and $\delta$.
Thus $\#\A_{n,p}\leq N_nK_{n^2,n-1}$.

If $1<p<n/(n-1)$, then $q=p/(p-1)>n$.
The bijection $J_p:\A_{n,p}\to\A_{n,q}$ gives the same bound.
This proves the theorem with $C_n=N_nK_{n^2,n-1}$.
\end{proof}

\section{Classification in dimensions two and three}\label{classification}

In this section $1<p<\infty$ and $p\ne2$. The following lemma is straightforward to verify.
\begin{lemma}[Auerbach bases of $\ell_p^2$]\label{lem:two-dimensional}
Up to equivalence, the Auerbach bases of $\ell_p^2$ are
\[
I_2\quad\text{and}\quad
H_2(p)=2^{-1/p}
\begin{pmatrix}1&1\\1&-1\end{pmatrix}.
\]
\end{lemma}

\begin{theorem}[Auerbach bases of $\ell_p^3$]
\label{thm:three-dimensional}
Let $1<p<\infty$, $p\ne2$. There are
exactly four equivalence classes of Auerbach matrices, represented by
\[
I_3,\qquad
L_2(p)=
\begin{pmatrix}
1&0&0\\
0&\alpha&\alpha\\
0&\alpha&-\alpha
\end{pmatrix},
\qquad
N_3(p)=
\begin{pmatrix}
\alpha&\alpha&0\\
\beta&-\beta&\alpha\\
\beta&-\beta&-\alpha
\end{pmatrix},
\]
and
\[
M_3(p)=(2+r^p)^{-1/p}
\begin{pmatrix}
r&1&1\\
1&r&-1\\
1&-1&r
\end{pmatrix},
\]
where
\[
\alpha=2^{-1/p},\qquad \beta=2^{-2/p},
\qquad r\in(0,1),\qquad r+r^{p-1}=1.
\]
\end{theorem}

\begin{proof}
Put \(\rho=p-1\). Let $B$ be an Auerbach matrix of order three.
If it has a row supported on exactly one coordinate, the other rows form an Auerbach
basis of the complementary coordinate plane, giving $I_3$
or $L_2(p)$. 

Suppose that $B$ has a zero entry but no row supported on
one coordinate. After signed permutations, choose a row
$(a,b,0)$ with $a,b>0$. The two biorthogonality equations
give $(a/b)^{\rho^2-1}=1$, hence $a=b=\alpha$.

The remaining rows have opposite first two coordinates.
Mutual biorthogonality forces their third coordinates to be
nonzero and of opposite signs. After sign changes, write them as
\[
(u,-u,v),\qquad(w,-w,-z),\qquad u,w,v,z>0.
\]
Then
\[
2uw^\rho=vz^\rho,\qquad 2u^\rho w=v^\rho z.
\]
Since $\rho\ne1$, these imply
$u/v=w/z=2^{-1/p}=\alpha$. Normalization gives
$v=z=\alpha$ and $u=w=\beta$, so $B$ is equivalent to $N_3(p)$.

It remains to consider full support. Make the first row
positive and each other row have exactly one negative entry.
These negative entries must occupy different columns,
since otherwise the corresponding duality pairing is positive.
Permute the columns so that the negative entries occur at
$(2,3)$ and $(3,2)$. Negating the first row and then the
second and third columns gives
\[
B=
\begin{pmatrix}
-a&b&c\\
d&-e&f\\
g&h&-i
\end{pmatrix},
\qquad a,b,c,d,e,f,g,h,i>0.
\]
Put $D=\det B$. The cofactor equation $Db^\rho=fg+di$
shows that $D>0$. Define
\[
T=aei,\quad X=afh,\quad Y=ceg,\quad Z=bdi,
\qquad U=bfg,\quad V=cdh.
\]
The identity $J_p(B)=B^{-T}$ gives
\begin{equation}\label{eq:three-cofactor-products}
D\bigl(|b_{jk}|^p\bigr)_{j,k=1}^3
=
\begin{pmatrix}
X-T&U+Z&V+Y\\
V+Z&Y-T&U+X\\
U+Y&V+X&Z-T
\end{pmatrix}.
\end{equation}
Thus $X,Y,Z>T$. Set
\[
x=X/T,\quad y=Y/T,\quad z=Z/T,\qquad
u=U/T,\quad v=V/T.
\]
Then $x,y,z>1$ and $uv=xyz$. Taking the ratio of
$X^p=(afh)^p$ and $T^p=(aei)^p$ in
\eqref{eq:three-cofactor-products} gives
\[
x^p=\frac{(u+x)(v+x)}{(y-1)(z-1)},
\]
or, using $uv=xyz$,
\[
x^\rho(y-1)(z-1)=yz+x+u+v,
\]
and cyclically. Subtracting the equations for $x$ and $y$
and cancelling $z-1$ yields
\[
\frac{x^\rho-1}{x-1}=\frac{y^\rho-1}{y-1}.
\]
For $\rho>0$, $\rho\ne1$, the function
\[
t\longmapsto\frac{t^\rho-1}{t-1}
=\rho\int_0^1(1+s(t-1))^{\rho-1}\,ds
\]
is strictly monotone on $(1,\infty)$. Hence $x=y=z$,
so $X=Y=Z$. Equation~\eqref{eq:three-cofactor-products}
now gives
\[
a=e=i,\qquad b=f=g,\qquad c=d=h.
\]
Consequently
\[
B=
\begin{pmatrix}
-a&b&c\\
c&-a&b\\
b&c&-a
\end{pmatrix},
\]
and biorthogonality between the first two rows gives
\[
cb^\rho=ac^\rho+ba^\rho,\qquad
bc^\rho=ab^\rho+ca^\rho.
\]
If $b\ne c$, a simultaneous interchange of the last two rows
and columns allows us to assume $b>c$. Subtracting and
rearranging gives
\[
a(b-c)(a^{\rho-1}-b^{\rho-1})
=c(a+b)(b^{\rho-1}-c^{\rho-1}).
\]
Strict monotonicity of $t^{\rho-1}$ forces $a>b$, contradicting
$bc^\rho=ab^\rho+ca^\rho$. Thus $b=c$.

Finally, with $r=a/b$, orthogonality and normalization give
\[
r+r^{p-1}=1,\qquad b=(2+r^p)^{-1/p}.
\]
Changing the signs of the first row and the last two columns
gives $M_3(p)$, completing the classification.
\end{proof}

\subsection{An infinite family in dimension four and higher} \label{sec:infinite 4dim}

For $|t|<1/2$, define
\[
 s=\sqrt{4-7t^2+4t^4},\qquad
 c=\sqrt{(s-t)/2},\qquad d=\sqrt{(s+t)/2}.
\]
Then $c,d>0$, $d^2-c^2=t$, and $cd=1-t^2$. Put
\[
 X(t)=\begin{pmatrix}
 -t&1&c&d\\1&t&d&-c\\d&-c&-t&1\\c&d&-1&-t
 \end{pmatrix},\qquad N(t)=3(1-t^2+t^4),\qquad
 B(t)=N(t)^{-1/4}X(t).
\]

Writing $X^{\circ3}$ for entrywise cubing, direct multiplication (after simplification) gives
\[
 X(X^{\circ3})^T=
 \begin{pmatrix}N&a&b&e\\-a&N&e&-b\\e&b&N&-a\\b&-e&a&N\end{pmatrix},
\]
where
\[
 \begin{split}
 a&=-t+t^3+cd^3-dc^3,\\
 b&=-td^3-c^3-ct^3+d,\\
 e&=-tc^3+d^3-c-dt^3.
 \end{split}
\]
Indeed the diagonal value is $t^4+1+c^4+d^4=N$.
The identities above give $a=0$. If $u=c^2$, then
$u(u+t)=(1-t^2)^2$, and elementary simplification gives
\[
 db=ce=t\bigl((1-t^2)^2-u(u+t)\bigr)=0.
\]
Thus $X(X^{\circ3})^T=NI_4$. This implies invertibility of $B$,
and gives $J_4(B)=B^{-T}$, so the rows of $B$ satisfy
\eqref{eq:auerbach}. Thus, for $0<t<1/2$, the matrices $B(t)$ form an infinite family of Auerbach bases of $l^4_4$.

For \(n>4\), the matrices \(B(t)\oplus I_{n-4}\) likewise belong to \(\A_{n,4}\), since \(J_4\) and inverse transposition respect direct sums.

\end{document}